\documentclass[10pt]{amsart}
 
\usepackage{amsthm,amsmath,amssymb}
\usepackage{graphicx}
\graphicspath{}
\RequirePackage{hyperref}
\RequirePackage{hypernat}

\newtheorem{theorem}{Theorem}
\newtheorem{lemma}{Lemma}

\newtheorem{corollary}{Corollary}

\newtheorem{definition}{Definition}
\theoremstyle{remark}

\numberwithin{equation}{section}

\def\ci{\perp\!\!\!\perp}

\makeatletter
\def\@strippedMR{}
\def\@scanforMR#1#2#3\endscan{%
  \ifx#1M\ifx#2R\def\@strippedMR{#3}%
  \else\def\@strippedMR{#1#2#3}%
  \fi\fi}
\renewcommand\MR[1]{\relax
  \ifhmode\unskip\spacefactor3000 \space\fi
  \@scanforMR#1\endscan
  \MRhref{\@strippedMR}{\@strippedMR}}
\makeatother

\title{PC algorithm for Gaussian copula graphical models}  

\author[N.~Harris]{Naftali Harris} 
\address{Department of Statistics, The University of Chicago, Chicago,
  IL, U.S.A.}
\email{naftali@uchicago.edu}

\author[M.~Drton]{Mathias Drton} 
\address{Department of Statistics, The University of Chicago, Chicago,
  IL, U.S.A.}
\email{drton@uchicago.edu}

\begin{document}

\begin{abstract}
  The PC algorithm uses conditional independence tests for model
  selection in graphical modeling with acyclic directed graphs.  In
  Gaussian models, tests of conditional independence are typically
  based on Pearson correlations, and high-dimensional consistency
  results have been obtained for the PC algorithm in this setting.  We
  prove that high-dimensional consistency carries over to the broader
  class of Gaussian copula or \textit{nonparanormal} models when using
  rank-based measures of correlation.  For graphs with bounded degree,
  our result is as strong as prior Gaussian results.  In simulations,
  the `Rank PC' algorithm works as well as the `Pearson PC' algorithm
  for normal data and considerably better for non-normal Gaussian
  copula data, all the while incurring a negligible increase of
  computation time.  Simulations with contaminated data show that rank
  correlations can also perform better than other robust estimates
  considered in previous work when the underlying distribution does
  not belong to the nonparanormal family.
\end{abstract}

\keywords{Copula, covariance matrix, graphical model, model selection,
  multivariate normal distribution, nonparanormal distribution} 

\maketitle

\section{Introduction}
\label{sec:introduction}

Let $G=(V,E)$ be an acyclic digraph with finite vertex set.  We will
typically write $v\to w\in E$ to indicate that $(v,w)$ is an edge in
$E$.  The digraph $G$ determines a statistical model for the joint
distribution of a random vector $X=(X_v)_{v \in V}$ by requiring that
$X$ satisfy conditional independence relations that are natural if the
edges in $E$ encode causal relationships among the random variables
$X_v$.  We refer the reader to
\cite{lauritzen:1996,pearl:2009,spirtes:2000} or
\cite[Chap.~3]{oberwolfach} for background on statistical modeling
with directed graphs.  As common in this field, we use the
abbreviation DAG (for `directed acyclic graph') to refer to acyclic
digraphs.

The conditional independences associated with the graph $G$ may be
determined using the concept of d-separation.  Since a DAG contains at
most one edge between any two nodes, we may define a path from a node
$u$ to a node $v$ to be a sequence of distinct nodes $(v_0, v_1,
\ldots , v_n)$ such that $v_0=u$, $v_n=v$ and for all $1\le k\le n$,
either $v_{k-1}\to v_k\in E$ or $v_{k-1}\leftarrow v_k\in E$.  Two
distinct nodes $u$ and $v$ are then said to be \emph{d-separated} by a
set $S \subset V \setminus \{v, u\}$ if every path from $u$ to $v$
contains three consecutive nodes $(v_{k-1},v_k,v_{k+1})$ for which one
of the following is true:
\begin{enumerate}
\item[(i)] The three nodes form a chain $v_{k-1} \to v_{k} \to v_{k+1}$, a
  chain $v_{k-1} \leftarrow v_{k} \leftarrow v_{k+1}$, or a fork
  $v_{k-1} \leftarrow v_{k} \to v_{k+1}$, and the middle node $v_k$ is
  in $S$.
\item[(ii)] The three nodes form a collider $v_{k-1} \to v_{k} \leftarrow
  v_{k+1}$, and neither $v_k$ nor any of its descendants is in $S$.
\end{enumerate}
Suppose $A, B, S$ are pairwise disjoint subsets of $V$.  Then $S$
d-separates $A$ and $B$ if $S$ d-separates any pair of nodes $a$ and
$b$ with $a \in A$ and $b \in B$.  Finally, the joint distribution of
the random vector $X=(X_v)_{v \in V}$ is \emph{Markov} with respect to
a DAG $G$ if $X_A$ and $X_B$ are conditionally independent given $X_S$
for any triple of pairwise disjoint subsets $A, B, S \subset V$ such
that $S$ d-separates $A$ and $B$ in $G$.  Here, $X_A$ denotes the
subvector $(X_v)_{v\in A}$.  It is customary to denote conditional
independence of $X_A$ and $X_B$ given $X_S$ by $X_A\ci X_B\,\vert\,
X_S$.

We will be concerned with the consistency of an algorithm for
inferring a DAG from data.  Graph inference is complicated by the fact
that two DAGs $G=(V,E)$ and $H=(V,F)$ with the same vertex set $V$ may
be \emph{Markov equivalent}, that is, they may possess the same
d-separation relations and, consequently, induce the same statistical
model.
To give an example, the graphs $u \to v \to w$ and $u\leftarrow
v\leftarrow w$ are Markov equivalent, but $u \to v \to w$ and $u\to
v\leftarrow w$ are not.  As first shown in \cite{verma:1991}, two DAGs
$G$ and $H$ are Markov equivalent if and only if they have the same
skeleton and the same unshielded colliders.  The \emph{skeleton} of a
digraph $G$ is the undirected graph obtained by converting each
directed edge into an undirected edge.  An \emph{unshielded collider}
is a triple of nodes $(u,v,w)$ that induces the subgraph $u\to
v\leftarrow w$, that is, there is no edge between $u$ and $w$.

Let $[G]$ be the Markov equivalence class of an acyclic digraph
$G=(V,E)$.  Write $E(H)$ for the edge set of a DAG $H$, and define the
edge set
\[
[E] = \bigcup_{H\in[G]} E(H).
\]
That is, $(v,w)\in[E]$ if there exists a DAG $H\in[G]$ with the edge
$v\to w$ in its edge set.  We interpret the presence of both $(v,w)$
and $(w,v)$ in $[E]$ as an undirected edge between $v$ and $w$.
Following the most closely related literature, we call the graph
$C(G)=(V,[E])$ the \emph{completed partially directed acyclic graph}
(CPDAG) for $G$, but other terminology such as the \emph{essential
  graph} is in use.  The graph $C(G)$ is partially directed as it may
contain both directed and undirected edges, and it is acyclic in the
sense of its directed subgraph having no directed cycles.  Two DAGs
$G$ and $H$ satisfy $C(G)=C(H)$ if and only if $[G]=[H]$, making the
CPDAG a useful graphical representation of a Markov equivalence class;
see \cite{andersson:1997,chickering:2002}.

The PC algorithm, named for its inventors Peter Spirtes and Clark
Glymour, uses conditional independence tests to infer a CPDAG from
data \cite{spirtes:2000}.  In its population version, the algorithm
amounts to a clever scheme to reconstruct the CPDAG $C(G)$ from
answers to queries about d-separation relations in the underlying DAG
$G$.  Theorem~\ref{thm:pc-background} summarizes the properties of the
PC algorithm that are relevant for the present paper.  For a proof of
the theorem as well as a compact description of the PC algorithm we
refer the reader to \cite{Kalisch:2007}.  Recall that the degree of a
node is the number of edges it is incident to, and that the degree of
a DAG $G$ is the maximum degree of any node, which we denote by $\deg(G)$.

\begin{theorem}
  \label{thm:pc-background}
  Given only the ability to check d-separation relations in a DAG $G$,
  the PC algorithm finds the CPDAG $C(G)$ by checking whether pairs of
  distinct nodes are d-separated by sets $S$ of cardinality $|S|\le
  \deg(G)$.
\end{theorem}

The joint distribution of a random vector $X=(X_v)_{v \in V}$ is
\emph{faithful} to the DAG $G$ if, for any triple of pairwise disjoint
subsets $A, B, S \subset V$, we have that $S$ d-separates $A$ and $B$
in $G$ if and only if $X_A\ci X_B\,\vert\, X_S$.  Under faithfulness,
statistical tests of conditional independence can be used to determine
d-separation relations in a DAG and lead to a sample version of the
PC algorithm that is applicable to data.

If $X$ follows the multivariate normal distribution
$\mathit{N}(\mu,\Sigma)$, with positive definite covariance matrix
$\Sigma$, then
\begin{equation}
  \label{eq:ci-gauss1}
  X_A\ci X_B\,\vert\, X_S \quad\iff\quad
  X_u\ci X_v\,\vert\, X_S \quad\forall\, u\in A,\: v\in B.
\end{equation}
Moreover, the pairwise conditional independence of $X_u$ and $X_v$ given $X_S$
is equivalent to the vanishing of the \emph{partial correlation}
$\rho_{uv\vert S}$, that is, the correlation obtained from the
bivariate normal conditional distribution of $(X_u,X_v)$ given $X_S$.
The iterations of the PC algorithm make use of the recursion
\begin{equation}\label{PartialWithRecursion}
  \rho_{uv \vert S} = \frac{\rho_{uv \vert S \setminus w} - \rho_{uw
  \vert S \setminus w} \rho_{vw \vert S \setminus w}}{\sqrt{\left(1 -
  \rho_{uw \vert S \setminus w}^2\right) \left(1 - \rho_{vw \vert S
  \setminus w}^2\right)}}, 
\end{equation}
where $w \in S$, and we define $\rho_{uv \vert \emptyset} = \rho_{uv}$
to be correlation of $u$ and $v$.  Our later theoretical analysis will
use the fact that 
\begin{equation} 
  \label{PartialWithInverses}
  \rho_{uv \vert S} = - \frac{\Psi^{-1}_{uv}}{\sqrt{\Psi^{-1}_{uu}
  \Psi^{-1}_{vv}}}, 
\end{equation}
where $\Psi = \Sigma_{(u,v,S),(u,v,S)}$ is the concerned principal
submatrix of $\Sigma$.  A natural estimate of $\rho_{uv\vert S}$ is
the sample partial correlation obtained by replacing $\Sigma$ with the
empirical covariance matrix of available observations.  Sample partial
correlations derived from independent normal observations have
favorable distributional properties \cite[Chap.~4]{anderson:2003},
which form the basis for the work of \cite{Kalisch:2007} who treat the
PC algorithm in the Gaussian context with conditional independence
tests based on sample partial correlations.  The main results in
\cite{Kalisch:2007} show high-dimensional consistency of the PC
algorithm, when the observations form a sample of independent normal
random vectors that are faithful to a suitably sparse DAG.

The purpose of this paper is to show that the PC algorithm has
high-dimensional consistency properties for a broader class of
distributions, when standard Pearson-type empirical correlations are
replaced by rank-based measures of correlations in tests of
conditional independence.  The broader class we consider comprises the
distributions with Gaussian copula.  Phrased in the terminology of
\cite{Liu:2009}, we consider \emph{nonparanormal} distributions.
Recall that a correlation matrix is a covariance matrix with all
diagonal entries equal to one.

\begin{definition}
  \label{def:npn}
  Let $f=(f_v)_{v \in V}$ be a collection of strictly increasing, but
  not necessarily continuous functions $f_v:\mathbb{R} \to
  \mathbb{R}$, and let $\Sigma \in \mathbb{R}^{V \times V}$ be a
  positive definite correlation matrix.  The nonparanormal
  distribution $\mathit{NPN}(f, \Sigma)$ is the distribution of the
  random vector $(f_v(Z_v))_{v \in V}$ for $(Z_v)_{v\in V} \sim
  \mathit{N}(0, \Sigma)$.
\end{definition}

Taking the functions $f_v$ to be affine shows that all multivariate
normal distributions are also nonparanormal.  If
$X\sim\mathit{NPN}(f,\Sigma)$, then the univariate marginal
distribution for a coordinate, say $X_v$, may have any continuous
cumulative distribution function $F$, as we may take $f_v = F^{-1}
\circ \Phi$, where $\Phi$ is the standard normal distribution
function.  

\begin{definition}
  \label{def:npn-model}
  The Gaussian copula graphical model $\mathit{NPN}(G)$ associated
  with a DAG $G$ is the set of all distributions
  $\mathit{NPN}(f,\Sigma)$ that are Markov with respect to $G$.
\end{definition}

Since the marginal transformations $f_v$ are deterministic,
the dependence structure in a nonparanormal distribution corresponds
to that in the underlying latent multivariate normal distribution.  In
other words, if $X \sim \mathit{NPN}(f, \Sigma)$ and
$Z\sim\mathit{N}(0,\Sigma)$, then it holds for any triple of pairwise
disjoint sets $A,B,S\subset V$ that
\begin{equation}
  X_A\ci X_B \,\vert\, X_S \;\iff\;
  Z_A\ci Z_B \,\vert\, Z_S.
\end{equation}
Hence, for two nodes $u$ and $v$ and a separating set $S \subset V
\setminus \{u, v\}$, it holds that
\begin{equation}
  \label{eq:npn-ci}
  X_u \ci X_v \,\vert\, X_S \;\iff\; 
  \rho_{uv \vert S} = 0,
\end{equation}
with $\rho_{uv\vert S}$ calculated from $\Sigma$ as in
(\ref{PartialWithRecursion}) or (\ref{PartialWithInverses}).  In light
of this equivalence, we will occasionally speak of a correlation
matrix $\Sigma$ being Markov or faithful to a DAG, meaning that the
requirement holds for any distribution
$\mathit{NPN}(f,\Sigma)$.

In the remainder of the paper we study the PC algorithm in the
nonparanormal context, proposing the use of Spearman's rank
correlation and Kendall's $\tau$ for estimation of the correlation
matrix parameter of a nonparanormal distribution.  In
Section~\ref{sec:rank-correlation}, we review how transformations of
Spearman's rank correlation and Kendall's $\tau$ yield accurate
estimators of the latent Gaussian correlations.  In particular, we
summarize tail bounds from \cite{Liu:2012}.
Theorem~\ref{ErrorProbsForPC} in Section~\ref{sec:rank-PC} gives our
main result, an error bound for the output of the PC algorithm when
correlations are used to determine nonparanormal conditional
independence.  In Corollary~\ref{ConsistencyOfPC}, we describe
high-dimensional asymptotic scenarios and suitable conditions that
lead to consistency of the PC algorithm.  The proof of
Theorem~\ref{ErrorProbsForPC} is given in Section~\ref{sec:proofs}.
Our simulations in Section~\ref{sec:simulations} make a strong case
for the use of rank correlations in the PC algorithm.  Some concluding
remarks are given in Section~\ref{sec:conclusion}.

\section{Rank correlations}
\label{sec:rank-correlation}

Let $(X,Y)$ be a pair of random variables, and let $F$ and $G$ be the
cumulative distribution functions of $X$ and $Y$, respectively.
Spearman's $\rho$ for the bivariate distribution of $(X,Y)$ is defined
as
\begin{equation}
  \rho^S = \text{Corr}\left(F(X), G(Y)\right),
\end{equation}
that is, it is the ordinary Pearson correlation between the quantiles
$F(X)$ and $G(Y)$.  Another classical measure of correlation is
Kendall's $\tau$, defined as
\begin{equation}
  \tau = \text{Corr}\left(\text{sign}\left(X - X'\right),
  \text{sign}\left(Y - Y'\right)\right) 
\end{equation}
where $(X',Y')$ is an independent copy of $(X, Y)$.  

Suppose $(X_1, Y_1), \ldots (X_n, Y_n)$ are independent pairs of
random variables, each pair distributed as $(X,Y)$.  Let
$\text{rank}(X_i)$ be the rank of $X_i$ among $X_1, \ldots, X_n$.  In the
nonparanormal setting, the marginal distributions are continuous so
that ties occur with probability zero, making ranks well-defined.  The
natural estimator of $\rho^S$ is  the sample correlation among
ranks, that is, 
\begin{align}\label{SpearmanEstimate}
  \hat\rho^S &= \frac{\frac{1}{n} \sum_{i = 1}^{n}
    \left(\frac{\text{rank}(X_i)}{n + 1} - \frac{1}{2}\right)
    \left(\frac{\text{rank}(Y_i)}{n + 1} -
      \frac{1}{2}\right)}{\sqrt{ \frac{1}{n}\sum_{i = 1}^{n}
        \left(\frac{\text{rank}(X_i)}{n + 1} - \frac{1}{2}\right)^2}
      \sqrt{
      \frac{1}{n}\sum_{i = 1}^{n} \left(\frac{\text{rank}(Y_i)}{n +
            1} - \frac{1}{2}\right)^2}}\\
  &=\label{FastSpearmanEstimate}
  1 - \frac{6}{n\left(n^2 - 1\right)}\sum_{i =
    1}^{n}\left(\text{rank}(X_i) - \text{rank}(Y_i)\right)^2,
\end{align}
which can be computed in $O(n \log n)$ time.  Kendall's $\tau$ may be
estimated by
\begin{equation}
  \hat{\tau} = \frac{2}{n\left(n-1\right)} \sum_{1 \le i < j \le n}
  \text{sign}\left(X_i - X_j\right) \text{sign}\left(Y_i -
  Y_j\right). 
\end{equation}
A clever algorithm using sorting and binary trees to compute
$\hat{\tau}$ in time $O(n \log n)$ instead of the naive $O(n^2)$ time 
has been developed by \cite{Christensen:2005}.

It turns out that simple trigonometric transformations of $\hat\rho^S$
and $\hat\tau$ are excellent estimators of the population Pearson
correlation for multivariate normal data. In particular,
\cite{Liu:2012} show that if $(X, Y)$ are bivariate normal with
$\text{Corr}(X, Y) = \rho$, then
\begin{equation}\label{SpearmanProbBound}
  \mathbb{P}\left(\left| 2 \sin\left(\frac{\pi}{6} \hat\rho^S\right) - \rho
  \right| > \epsilon \right) \;\le\; 2 \exp \left(-\frac{2}{9 \pi^2} n
  \epsilon^2 \right) 
\end{equation}
and
\begin{equation}\label{KendallProbBound}
  \mathbb{P}\left(\left| \sin\left(\frac{\pi}{2} \hat{\tau}\right) - \rho
  \right| > \epsilon \right) \;\le\; 2 \exp \left(-\frac{2}{\pi^2} n
  \epsilon^2 \right) .
\end{equation}

Clearly, $\hat{\rho}^S$ and $\hat{\tau}^K$ depend on the observations
$(X_1, Y_1), \ldots (X_n, Y_n)$ only through their ranks.  Since ranks
are preserved under strictly increasing functions,
\eqref{SpearmanProbBound} and \eqref{KendallProbBound} still hold if
$(X, Y)\sim\mathit{NPN}(f,\Sigma)$ with Pearson correlation
$\rho=\Sigma_{xy}$ in the underlying
latent bivariate normal distribution.  Throughout the rest of this
paper, we will assume that we have some estimator $\hat{\rho}$ of
$\rho$ which has the property that, for nonparanormal data,
\begin{equation}
  \label{eq:generic-corr-estimator}
  \mathbb{P}(\left| \hat{\rho} - \rho \right| > \epsilon) \;<\; A \exp \left(-B n
  \epsilon^2 \right) 
\end{equation}
for fixed constants $0 < A, B < \infty$.  As just argued, the
estimators considered in \eqref{SpearmanProbBound} and
\eqref{KendallProbBound} both have this property. 

When presented with multivariate observations from a distribution
$\mathit{NPN}(f,\Sigma)$, we apply the estimator
from~(\ref{eq:generic-corr-estimator}) to every pair of coordinates to
obtain an estimator $\hat\Sigma$ of the correlation matrix parameter.
Plugging $\hat\Sigma$ into \eqref{PartialWithRecursion} or
equivalently into \eqref{PartialWithInverses} gives partial
correlation estimators that we denote $\hat\rho_{uv\vert S}$.

\section{Rank PC algorithm}
\label{sec:rank-PC}

Based on the equivalence~(\ref{eq:npn-ci}), we may use the rank-based
partial correlation estimates $\hat\rho_{uv\vert S}$ to test
conditional independences.  In other words, we conclude that
\begin{equation}\label{CITest}
  X_u \ci X_v \vert X_S \quad\iff\quad \left| \hat{\rho}_{uv \vert S}
  \right| \le \gamma,
\end{equation}
where $\gamma \in \left[0, 1\right]$ is a fixed threshold.  We will
refer to the PC algorithm that uses the conditional independence tests
from~(\ref{CITest}) as the `Rank PC' (RPC) algorithm.  We write $\hat
C_{\gamma}(G)$ for the output of the RPC algorithm with tuning parameter
$\gamma$.

The RPC algorithm consist of two parts.  The first part computes the
correlation matrix $\hat{\Sigma} = (\hat{\rho}_{uv})$ in time $O(p^2 n
\log n)$, where $p:=|V|$.  This computation takes $O(\log n)$ longer
than its analogue under use of Pearson correlations.  The second part
of the algorithm is independent of the type of correlations involved.
It determines partial correlations and performs graphical operations.
For an accurate enough estimate of a correlation matrix $\Sigma$ that
is faithful to a DAG $G$, this second part takes $O(p^{\deg(G)})$ time
in the worst case, but it is often much faster; compare
\cite{Kalisch:2007}.  For high-dimensional data with $n$ smaller than
$p$, the computation time for RPC is dominated by the second part, the
PC-algorithm component.  Moreover, in practice, one may wish to apply
RPC for several different values of $\gamma$, in which case the
estimate $\hat\Sigma$ needs to be calculated only once. As a result,
Rank PC takes only marginally longer to compute than Pearson PC for
high-dimensional data.

What follows is our main result about the correctness of RPC, which
we prove in Section~\ref{sec:proofs}.  For a correlation matrix
$\Sigma\in\mathbb{R}^{V\times V}$, let 
\begin{equation}\label{RegularityAssumption0}
  c_{\min}(\Sigma) := \min \left\{| \rho_{uv \vert S}| \::\:   u,v\in
    V,\, S\subseteq V\setminus\{u,v\},\,\rho_{uv
      \vert S} \ne 0\right\}
\end{equation}
be the minimal magnitude of any non-zero partial correlation, and let
$\lambda_{\min}(\Sigma)$ be the minimal eigenvalue.  Then for any
integer $q\ge 2$, let
\begin{align}\label{RegularityAssumption}
  c_{\min}(\Sigma,q) &:= \min \left\{\, c_{\min}(\Sigma_{I,I})\::\: I
    \subseteq V,\, |I| = 
    q\,\right\}, \quad\text{and}\\
\label{EigMinAssumption}
  \lambda_{\min}(\Sigma,q) &:= \min\left\{
  \lambda_{\min}(\Sigma_{I,I}) \::\: I \subseteq V,\, |I| =
    q \,\right\}
\end{align}
be the minimal magnitude of a non-zero partial correlation and,
respectively, the minimal eigenvalue of any $q\times q$ principal
submatrix of $\Sigma$.  Note that if $I\subset J$ then
$c_{\min}(\Sigma_{I,I})\le c_{\min}(\Sigma_{J,J})$ and
$\lambda_{\min}(\Sigma_{I,I})\le \lambda_{\min}(\Sigma_{J,J})$.

\begin{theorem}[Error bound for RPC-algorithm]
  \label{ErrorProbsForPC} 
  Let $X_1,\dots, X_n$ be a sample of independent observations drawn
  from a nonparanormal distribution $\mathit{NPN}(f,\Sigma)$ that is
  faithful to a DAG $G$ with $p$ nodes.  For $q:=\deg(G) + 2$, let
  $c:=c_{\min}(\Sigma,q)$ and $\lambda:= \lambda_{\min}(\Sigma,q)$.
  If $n> q$, then there exists a threshold $\gamma \in 
  [0, 1]$ for which
  \begin{equation*}
    \mathbb{P}\big(\,\hat{C}_{\gamma}(G) \ne C(G)\,\big) \;\le\; \frac{A}{2} p^2 \exp
    \left(-\frac{B \lambda^4 n c^2}{36 q^2} \right),
  \end{equation*}
  where $0<A,B<\infty$ are the constants from~(\ref{eq:generic-corr-estimator}).
\end{theorem}

We remark that while all subsets of size $q$ appear in the
definitions in~(\ref{RegularityAssumption}) and
(\ref{EigMinAssumption}), our proof of Theorem~\ref{ErrorProbsForPC}
only requires the corresponding minima over those principal
submatrices that are actually inverted in the run of the PC-algorithm.

From the probability bound in Theorem~\ref{ErrorProbsForPC}, we may
deduce high-dimensional consistency of RPC.  For two positive
sequences $(s_n)$ and $(t_n)$, we write $s_n=O(t_n)$ if 
$s_n \le M t_n$, and $s_n=\Omega(t_n)$ if $s_n\ge M t_n$ for a constant
$0<M<\infty$.

\begin{corollary}[Consistency of RPC-algorithm]\label{ConsistencyOfPC}
  Let $(G_n)$ be a sequence of DAGs.  Let $p_n$ be the number of nodes
  of $G_n$, and let $q_n=\deg(G_n)+2$.  Suppose $(\Sigma_n)$ is a
  sequence of $p_n\times p_n$ correlation matrices, with $\Sigma_n$
  faithful to $G_n$. Suppose further that there are constants $0\le
  a,b,d,f<1$ that govern the growth of the graphs as
  \begin{align*}
    \log p_n &= O(n^a), & q_n &= O(n^b),
  \end{align*}
  and minimal signal strengths and eigenvalues as
  \begin{align*}
    c_{\min}(\Sigma_n,q_n)&=\Omega(n^{-d}), & \lambda_{\min}(\Sigma_n,q_n) &= \Omega(n^{-f}).
  \end{align*}
  If $a+2b +2d + 4f < 1$, then there exists a sequence of
  thresholds $\gamma_n$ for which
  \[
  \lim_{n \rightarrow \infty}
  \mathbb{P}\big(\,\hat C_{\gamma_n}(G_n) = C(G_n)\,\big) = 1,
  \]
  where $\hat C_{\gamma_n}(G_n)$ is the output of the RPC algorithm for a
  sample of independent observations $X_1,\dots,X_n$ from a
  nonparanormal distribution $\mathit{NPN}(\,\cdot\,,\Sigma_n)$.
\end{corollary}
\begin{proof}
  By Theorem \ref{ErrorProbsForPC}, for large enough $n$, we can pick
  a threshold $\gamma_n$ such that
  \begin{equation}
    \mathbb{P}(\hat{C}_{\gamma_n}(G_n) \ne C(G_n) \le A' \exp \left(2n^{a}-B' n^{1 - 2b-2d - 4f } \right)
  \end{equation}
  for constants $0<A',B'<\infty$.  The bound goes to zero
  if $1-2b - 2d - 4f > a$.
\end{proof}

As previously mentioned, \cite{Kalisch:2007} prove a similar
consistency result in the Gaussian case.  Whereas our proof consists
of propagation of errors from correlation to partial correlation
estimates, their proof appeals to Fisher's result that under
Gaussianity, sample partial correlations follow the same type of
distribution as sample correlations when the sample size is adjusted
by subtracting the cardinality of the conditioning set
\cite[Chap.~4]{anderson:2003}.  It is then natural to work with a
bound on the partial correlations associated with small conditioning
sets.  More precisely, \cite{Kalisch:2007} assume that there is a
constant $0\le M<1$ such that for any $n$, the partial correlations
$\rho_{uv\vert S}$ of the matrix $\Sigma_n$ satisfy
\begin{equation} \label{BoundedPartials}
 | \rho_{uv \vert S}| \le M \qquad \forall \: u,v\in V,\;
 S\subseteq V\setminus\{u,v\},\, |S| \le q_n.
\end{equation}
It is then no longer necessary to involve the minimal eigenvalues
from~(\ref{EigMinAssumption}).  The work in \cite{Kalisch:2007} is
thus free of an analogue to our constant $f$.  Stated for the case of
polynomial growth of $p_n$ (with $a=0$), their result gives
consistency when $b+2d<1$, whereas our condition requires $2b+2d<1$
even if $f=0$.  (Note that our constant $b$ corresponds to $1-b$ in
\cite{Kalisch:2007}.)

In the important special case of bounded degree, however, our
nonparanormal result is just as strong as the previously established
Gaussian consistency guarantee.  Staying with polynomial growth of
$p_n$, i.e., $a=0$, suppose the sequence of graph degrees $\deg(G_n)$
is indeed bounded by a fixed constant, say $q_0-2$.  Then clearly, $b
= 0$.  Moreover, the set of correlation matrices of size $q_0$
satisfying \eqref{BoundedPartials} with $q_n= q_0$ is compact.  Since
the smallest eigenvalue is a continuous function, the infimum of all
eigenvalues of such matrices is achieved for some invertible matrix.
Hence, the smallest eigenvalue is bounded away from zero, and we
conclude that $f = 0$.  Corollary~\ref{ConsistencyOfPC} thus implies
consistency if $2d <1$, or if $d < \frac{1}{2} = \frac{1-b}{2}$,
precisely as in \cite{Kalisch:2007}.  (No generality is lost by
assuming $a=0$; in either one of the compared results this constant is
involved solely in a union bound over order $p^2$ events.)

\section{Proof of the error bound}
\label{sec:proofs}

In this section, we prove the error bound in
Theorem~\ref{ErrorProbsForPC}.  Our argument starts from a uniform
bound on the error in our estimate $\hat\Sigma$.  Then we analyze how
this error propagates to the partial correlation estimates
$\hat\rho_{uv\vert S}$, giving again a uniform error bound.  We begin
by proving three lemmas about the error propagation.

The first lemma invokes classical results on error propagation in matrix
inversion.  
Let $\lVert A \rVert$ denote the spectral norm of a matrix
$A=(a_{ij})=\mathbb{R}^{q\times q}$, that is, $\lVert A \rVert^2$ is
the maximal eigenvalue of $A^TA$.  
Write the $l_\infty$ vector
norm of $A$ as
\[
\lVert A\rVert_\infty = \max_{1\le i,j\le q} |a_{ij}|.
\]

\begin{lemma}[Errors in matrix inversion]\label{ErrorBoundsForInverse}
  Suppose $\Sigma \in \mathbb{R}^{q \times q}$ is an invertible matrix
  with minimal eigenvalue $\lambda_{\min}$.  If $E \in \mathbb{R}^{q
    \times q}$ is a matrix of errors with $\lVert E\rVert_\infty <
  \epsilon < \lambda_{\min}/q$, then $\Sigma + E$ is invertible and
  \begin{equation*}
    \lVert \left(\Sigma + E\right)^{-1} - \Sigma^{-1} \rVert_\infty
     \le \frac{q \epsilon /\lambda_{\min}^2}{1 - q \epsilon /\lambda_{\min}}.
  \end{equation*}
\end{lemma}
\begin{proof}
  First, note that 
  \begin{equation}
    \label{eq:norm-equiv}
    \|E\|_\infty\le \|E\|\le q\|E\|_\infty; 
  \end{equation}
  see entries $(2,6)$ and $(6,2)$ in the table on p.~314 in
  \cite{Horn:1990}.  Using the submultiplicativity of a matrix norm, the second
  inequality in~(\ref{eq:norm-equiv}), and our assumption on
  $\epsilon$, we find that
  \begin{equation}
    \label{eq:mx-invert-2}
    \lVert E\Sigma^{-1}\rVert\le \lVert \Sigma^{-1}\rVert\cdot \lVert
    E\rVert < \frac{ q \epsilon}{\lambda_{\min}} < 1.
  \end{equation}
  As discussed in \cite[Sect.~5.8]{Horn:1990}, this implies that $I+E\Sigma^{-1}$ and thus also $\Sigma+E$ is invertible.  Moreover, 
  by the first inequality in~(\ref{eq:norm-equiv}) and inequality
  (5.8.2) in \cite{Horn:1990}, we obtain that
  \begin{equation}
    \label{eq:mx-invert-1}
    \lVert \left(\Sigma + E\right)^{-1} - \Sigma^{-1} \rVert_\infty\le
    \lVert \left(\Sigma + E\right)^{-1} - \Sigma^{-1} \rVert \le
    \lVert \Sigma^{-1}\rVert\cdot \frac{\lVert E\Sigma^{-1}\rVert}{1-\lVert
      E\Sigma^{-1}\rVert}.
  \end{equation}
  Since the function $x\mapsto x/(1-x)$ is increasing for $x<1$, our
  claim follows from the fact that $\lVert
  \Sigma^{-1}\rVert=1/\lambda_{\min}$ and the inequality $\lVert
  E\Sigma^{-1}\rVert < q \epsilon/\lambda_{\min}$
  from~(\ref{eq:mx-invert-2}).
\end{proof}

\begin{lemma}[Diagonal of inverted correlation matrix]\label{InverseDiagonals}
  If $\Sigma\in\mathbb{R}^{q\times q}$ is a positive definite
  correlation matrix, then the diagonal entries of
  $\Sigma^{-1}=(\sigma^{ij})$ satisfy $\sigma^{ii}\ge 1$.
\end{lemma}
\begin{proof}
  The claim is trivial for $q = 1$.  So assume $q \ge 2$.  By
  symmetry, it suffices to consider the entry $\sigma^{qq}$, and we
  partition the matrix as
  \begin{equation}
    \Sigma = 
      \begin{pmatrix}
       A   & b \\
       b^T & 1 \\
       \end{pmatrix}
    \end{equation}
    with $A \in \mathbb{R}^{(q - 1) \times (q - 1)}$ and $b \in
    \mathbb{R}^{q-1}$.  By the Schur complement formula for the
    inverse of a partitioned matrix,
    \[
    \sigma^{qq} = \frac{1}{1-b^TA^{-1}b};
    \]
    compare \cite[\S0.7.3]{Horn:1990}.  Since $A$ is positive
    definite, so is $A^{-1}$.  Hence, $b^TA^{-1}b\ge 0$. Since $\Sigma^{-1}$ is
    positive definite, $\sigma^{qq}$ cannot be negative, and so we deduce
    that $\sigma_{qq}\ge 1$, with equality if and only if $b=0$.
\end{proof}

The next lemma treats the error propagation from the inverse of a
correlation matrix to partial correlations.

\begin{lemma}[Error in partial correlations]\label{2rOver1-r}
  Let $A=(a_{ij})$ and $B=(b_{ij})$ be symmetric $2\times 2$ matrices.
  If $A$ is positive definite with $a_{11},a_{22}\ge 1$ and
  $\|A-B\|_\infty < \delta < 1$, then
  \[
  \bigg| \frac{a_{12}}{\sqrt{a_{11}a_{22}}} -
  \frac{b_{12}}{\sqrt{b_{11}b_{22}}} \bigg| <
  \frac{2\delta}{1-\delta}.
  \]
\end{lemma}

\begin{proof}
  Without loss of generality, suppose $a_{12} \ge 0$.  Since
  $\|A-B\|_\infty < \delta$,
  \begin{multline*}
    \frac{b_{12}}{\sqrt{b_{11}b_{22}}} -
      \frac{a_{12}}{\sqrt{a_{11}a_{22}}} <
      \frac{a_{12}+\delta}{\sqrt{\left(a_{11}-\delta\right)
          \left(a_{22}-\delta\right)}} 
      - \frac{a_{12}}{\sqrt{a_{11}a_{22}}} \\ 
      =
      \frac{\delta}{\sqrt{\left(a_{11}-\delta\right)
          \left(a_{22}-\delta\right)}}
      +
      a_{12}\left(\frac{1}{\sqrt{\left(a_{11}-\delta\right)
            \left(a_{22}-\delta\right)}}
      - \frac{1}{\sqrt{a_{11}a_{22}}}\right).
  \end{multline*}
  Using that $a_{11},a_{22}\ge 1$ to bound the first term and
  $a_{12}^2<a_{11}a_{22}$ to bound the second term, we obtain that
  \begin{align*}
    \frac{b_{12}}{\sqrt{b_{11}b_{22}}} -
    \frac{a_{12}}{\sqrt{a_{11}a_{22}}} &< \frac{\delta}{1-\delta} +
    \sqrt{a_{11}a_{22}}
    \left(\frac{1}{\sqrt{\left(a_{11}-\delta\right)\left(a_{22}-\delta\right)}}
      - \frac{1}{\sqrt{a_{11}a_{22}}}\right) \\ 
    &= \frac{\delta}{1-\delta} +
    \left(\sqrt{\frac{a_{11}}{a_{11}-\delta}\cdot
        \frac{a_{22}}{a_{22}-\delta}} - 1\right) .
  \end{align*}
  Since the function $x\mapsto x/(x-\delta)$ is decreasing, we may
  use our assumption that $a_{11},a_{22}\ge 1$ to get the bound
  \begin{equation*}
    \frac{b_{12}}{\sqrt{b_{11}b_{22}}} -
    \frac{a_{12}}{\sqrt{a_{11}a_{22}}} < \frac{\delta}{1-\delta} +
    \left(\sqrt{\frac{1}{1-\delta}\cdot \frac{1}{1-\delta}} - 1\right)
    = \frac{2\delta}{1-\delta}
  \end{equation*}
  A similar argument yields that
  \begin{equation}
    \frac{a_{12}}{\sqrt{a_{11}a_{22}}} -
    \frac{b_{12}}{\sqrt{b_{11}b_{22}}} < \frac{2\delta}{1+\delta}, 
  \end{equation}
  from which our claim follows.
\end{proof}

We are now ready to prove our main result.

\begin{proof}[Proof of Theorem \ref{ErrorProbsForPC}]
  We will show that our claimed probability bound for the event $\hat
  C_{\gamma}(G)\not= C(G)$ holds when the threshold in the RPC algorithm
  is $\gamma = c/2$.  By Theorem~\ref{thm:pc-background}, if
  all conditional independence tests for conditioning sets of size
  $|S|\le \deg(G)$ make correct decisions, then the output of the RPC
  algorithm $\hat C_{\gamma}(G)$ is equal to the CPDAG $C(G)$.  When
  $\gamma=c/2$, the conditional independence test accepts a
  hypothesis $X_u \ci X_v \vert X_S$ if and only if $| \hat{\rho}_{uv
    \vert S}| < \gamma = c/2$.  Hence,  the test makes a correct decision if $| \hat{\rho}_{uv \vert S} -
  \rho_{uv \vert S} | < c/2$
  because all non-zero partial correlations for $|S|\le \deg(G)$ are
  bounded away from zero by $c$; recall~(\ref{RegularityAssumption0})
  and~(\ref{RegularityAssumption}).  It remains to show, using the
  error analysis from Lemmas~\ref{ErrorBoundsForInverse} and~\ref{2rOver1-r}, that the event $| \hat{\rho}_{uv \vert S} -
  \rho_{uv \vert S} | \ge c/2$ occurs with small enough probability
  when $|S|\le \deg(G)$.
  
  Suppose our correlation matrix estimate
  $\hat\Sigma=(\hat\rho_{uv})$ satisfies $\lVert \hat\Sigma -
  \Sigma\rVert_\infty<\epsilon$ for
  \begin{equation}\label{epsdef}
    \epsilon = \frac{c \lambda^2}{(4+c) q + \lambda c q} > 0.
  \end{equation}
  Choose any two nodes $u,v\in V$ and a set $S\subseteq
  V\setminus\{u,v\}$ with $|S|\le \deg(G)=q-2$.  Let $I=\{u,v\}\cup S$
  and define $\Psi = \Sigma_{I,I}$ and $\hat{\Psi} =
  \hat{\Sigma}_{I,I}$, two principal submatrices of size at most $q$.
  For the choice of $\epsilon$ from~(\ref{epsdef}), the assumptions of
  Lemma~\ref{ErrorBoundsForInverse} hold and we deduce that
  $\hat{\Psi}$ is invertible, with
  \begin{equation}\label{eq:InverseBounds}
     \lVert \hat{\Psi}^{-1} - \hat{\Psi}^{-1} \rVert \;<\;
     \frac{q \epsilon /\lambda^2}{1 - q \epsilon /\lambda} 
    = \frac{q c}{(4 + c) q + \lambda c q - \lambda c q} 
    = \frac{c}{4+c}.
  \end{equation}
  By Lemma \ref{InverseDiagonals}, all diagonal entries of $\Psi^{-1}$
  are equal to one or greater, and so we can apply
  Lemma~\ref{2rOver1-r} with (\ref{eq:InverseBounds}).  Letting
  $\delta=c/(4+c)$, we get that
  \begin{equation*}
    | \hat{\rho}_{uv\vert S} - \rho_{uv \vert S} | \;=\; 
    \left| \frac{\hat{\Psi}^{-1}_{uv}}{\sqrt{\hat{\Psi}^{-1}_{uu}
    \hat{\Psi}^{-1}_{vv}}} -
    \frac{\Psi^{-1}_{uv}}{\sqrt{\Psi^{-1}_{uu} \Psi^{-1}_{vv}}}
    \right|  
    \;<\; \frac{2 \delta}{1-\delta} 
    = \frac{c}{2}.
  \end{equation*}
  Therefore, $\lVert \hat\Sigma - \Sigma\rVert_\infty<\epsilon$
  implies that our tests decide all conditional independences
  correctly in the RPC algorithm.
  
  Next, using~(\ref{eq:generic-corr-estimator}) and a union bound, we
  find that
  \begin{align*}
    \mathbb{P}\left(\hat C_{\gamma}(G) \ne C(G)\right) &\le
    \mathbb{P}\left( | \hat{\Sigma}_{uv} - \Sigma_{uv}|\ge \epsilon\:
      \text{ for some } u,v \in V \right) \\
    &\le A \frac{p(p-1)}{2} \exp \left(-B n \epsilon^2 \right) .
  \end{align*}
  Plugging in the definition of $\epsilon$ gives the claimed inequality
  \begin{equation*}
    \mathbb{P}\left(\hat C_{\gamma}(G) \ne C(G)\right) \le
    \frac{A}{2}p^2 \exp \left(-\frac{B \lambda^4 n c^2}{\left((4+c) q + \lambda c q\right)^2} \right) 
    \le \frac{A}{2} p^2 \exp \left(-\frac{B \lambda^4 n c^2}{36 q^2}
    \right)
  \end{equation*}
  because $c \le 1$ and $\lambda \le 1$.  The inequality $c\le 1$
  holds trivially because partial correlations are in $[-1,1]$.  The
  inequality $\lambda\le 1$ holds because a $q\times q$ correlation
  matrix has trace $q$, this trace is equal to the sum of the $q$
  eigenvalues, and $\lambda$ is the minimal eigenvalue.
\end{proof}

\section{Simulations}
\label{sec:simulations}

In this section we evaluate the finite-sample performance of the RPC
algorithm in simulations.  We compare RPC to two other versions of the
PC-algorithm: (i) `Pearson-PC', by which we mean the standard approach
of using sample partial correlations to test Gaussian conditional
independences, and (ii) `$Q_n$-PC', which is based on a robust
estimator of the covariance matrix and was considered in
\cite{Kalisch:2008}.  All our computations are done with the
\texttt{pcalg} package for R \cite{pcalg}.

The Gaussian conditional independence tests in the \texttt{pcalg}
package (and other software such as \texttt{Tetrad
  IV}\footnote{http://www.phil.cmu.edu/projects/tetrad}) use a level
$\alpha \in \left[0, 1\right]$ and decide that 
\begin{equation}\label{KBCITest}
  X_u \ci X_v \vert X_S \;\iff\; \sqrt{n - \left|S\right| - 3} \left|
  \frac{1}{2} \log \left(\frac{1 + \hat{\rho}_{uv \vert S}}{1 -
  \hat{\rho}_{uv \vert S}}\right) \right| \le \Phi^{-1}\left(1 -
  \alpha/2\right). 
\end{equation}
If the observations are multivariate normal and $\hat\rho_{uv\vert S}$
are sample partial correlations then $\alpha$ is an asymptotic
significance level for the test.  The test in (\ref{KBCITest}) is
equivalent to our earlier setup of conditional independence tests in
\eqref{CITest}, with the exception of the sample size adjustment from
$n$ to $n - \left|S\right| - 3$.  This adjustment is motivated by
classical large-sample bias-correction theory for Fisher's z-transform
of sample correlations; compare \cite{anderson:2003}.  We show in the
appendix that the adjustment has no affect on the consistency result
we proved in Corollary \ref{ConsistencyOfPC}.


Following the setup of \cite{Kalisch:2007}, we simulate random DAGs
and sample from probability distributions faithful to them.  Fix a
sparsity parameter $s\in [0,1]$ and enumerate the vertices as
$V=\{1,\dots,p\}$.  Then we generate a DAG by including the edge $u\to
v$ with probability $s$, independently for each pair $(u,v)$ with $1
\le u < v \le p$.  In this scheme, each node has the same expected
degree, namely, $(p-1)s$.

Given a DAG $G=(V,E)$, let $\Lambda=(\lambda_{uv})$ be a $p\times p$
matrix with $\lambda_{uv}=0$ if $u\to v\not\in E$.  Furthermore, let
$\epsilon=(\epsilon_1,\dots,\epsilon_p)$ be a vector of independent
random variables.  Then the random vector $X$ solving the equation
system
\begin{equation}
  \label{eq:struct-eqn}
X = \Lambda X+ \epsilon
\end{equation}
is well-known to be Markov with respect to $G$.
Here, we draw the edge coefficients $\lambda_{uv}$, $u\to v\in E$,
independently from a uniform distribution on the interval $(0.1,1)$.
For such a random choice, with probability one, the vector $X$ solving
\eqref{eq:struct-eqn} is faithful with respect to $G$.
We consider three different types of data:
\begin{enumerate}
\item[(i)] multivariate normal observations, which we generate by
  taking $\epsilon$ in~(\ref{eq:struct-eqn}) to have independent
  standard normal entries;
\item[(ii)] observations from the Gaussian copula model, for which we
  transform the marginals of the normal random vectors from (i) to an
  $F_{1,1}$-distribution; 
\item[(iii)] contaminated data, for which we generate the entries of
  $\epsilon$ in~(\ref{eq:struct-eqn}) as independent draws from a
  80-20 mixture between a standard normal and a standard Cauchy
  distribution.
\end{enumerate}
The contaminated distributions in (iii) do not belong
to the nonparanormal class.  

We consider the RPC algorithm in the version that uses Spearman
correlations as in \eqref{SpearmanProbBound}; the results for
Kendall's $\tau$ are similar.  When comparing graph estimates, we use
the Structural Hamming Distance (SHD) as a measure of distance.  The
SHD is the number of edge insertions, deletions, and reorientations
required to transform one graph to another.  An undirected edge is
counted as a single edge.

For the simulations we consider each combination of
\[
p \in \{10, 22, 46, 100\} \quad\text{and}\quad
n \in \{32, 100, 316, 1000, 3162\},
\]
and choose the expected degree as either 3 or 6.  In each case, we
draw 240 random graphs and then generate samples of $n$ observations.
For the tuning parameter in (\ref{KBCITest}), we consider a fixed grid,
namely,
\begin{equation*}
  \log_{10} \alpha \in\{
  -7, -6, -5, -4.25, -3.5, -2.75, -2, -1.5, -1, -0.75\}.
\end{equation*}
For each of the resulting combinations, we run each of the three
considered versions of the PC algorithm, retaining the result for the
best choice among 7 values for $\alpha$, best in terms of lowest
average SHD to the true underlying DAG for a given combination.  In
Figures~\ref{fig:normal},~\ref{fig:F11} and \ref{fig:Cauchymixture}, we
plot the resulting SHDs against the sample size $n$.

\begin{figure}[t]
  \centering
  \includegraphics[scale=0.65]{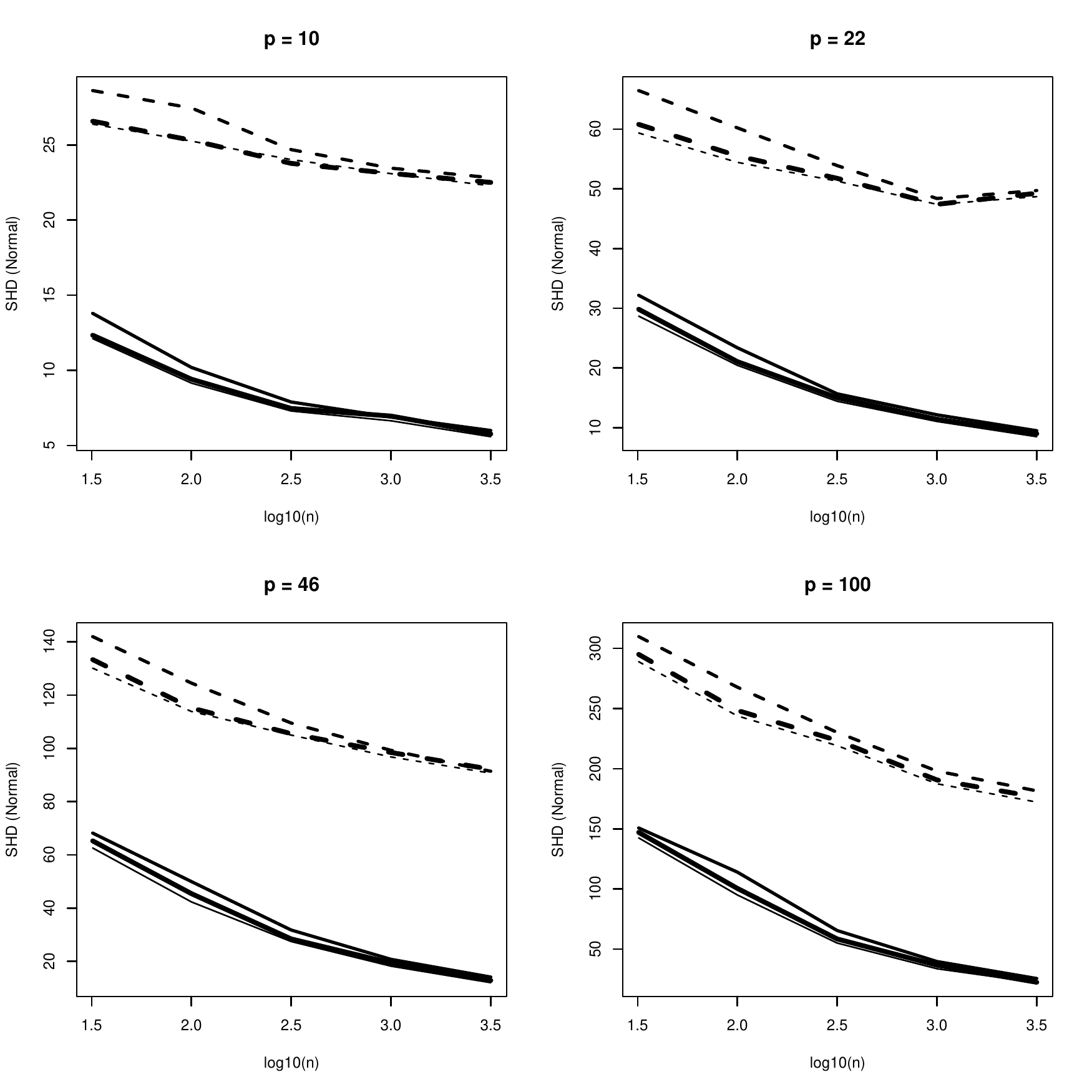} 
  \caption{  \label{fig:normal}
    Structural Hamming distances for normal data, graphs with expected
    degree 3 (solid lines) and 6 (dotted lines), and three versions of
    the PC algorithm: Pearson-PC (thin lines), $Q_n$-PC (medium lines)
    and RPC using Spearman's $\rho$ (thick lines).}
\end{figure}

\begin{figure}[t]
  \centering
  \includegraphics[scale=0.65]{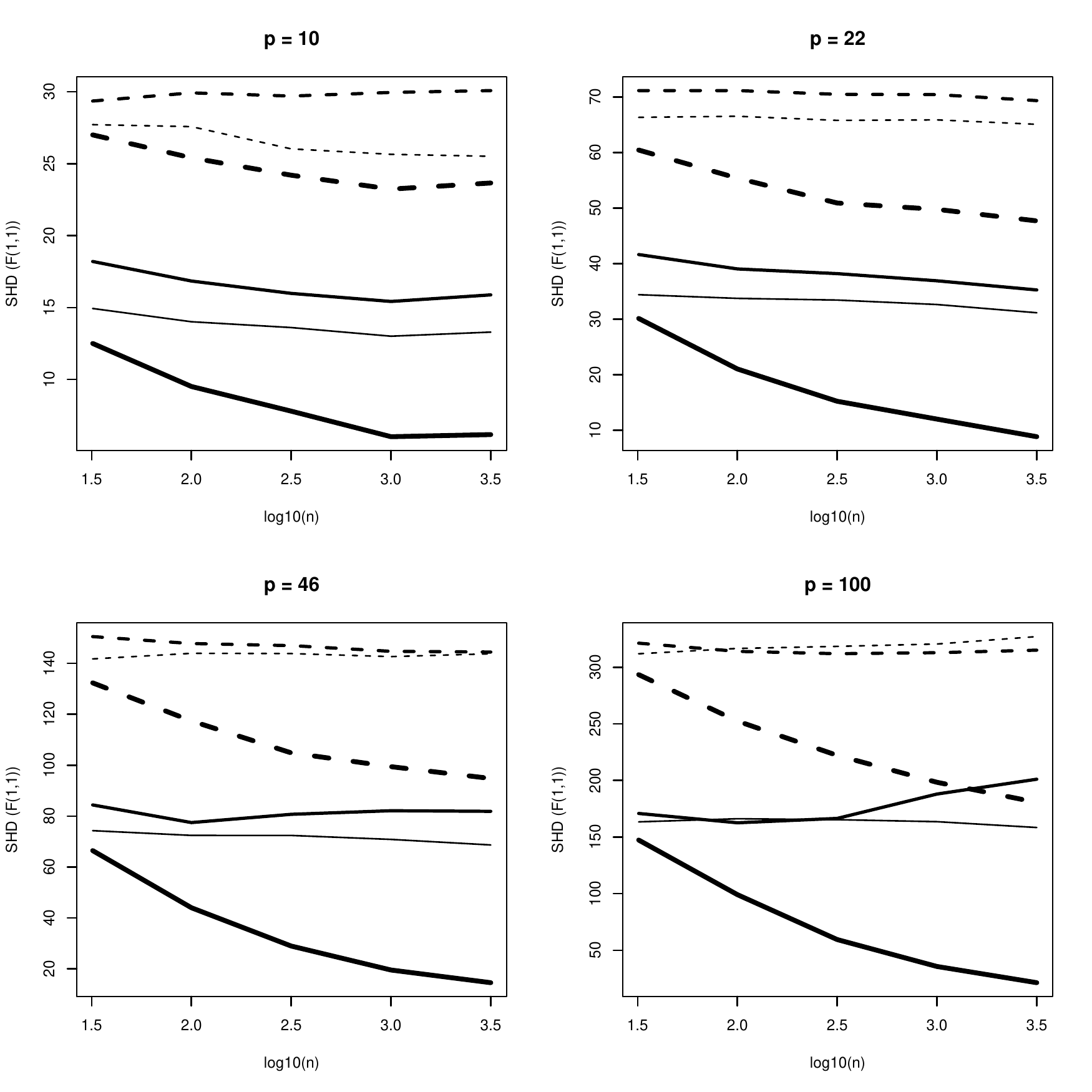} 
  \caption{  \label{fig:F11}
    Structural Hamming distances for Gaussian copula data with
    $F_{1,1}$ marginals, graphs with expected degree 3 (solid lines) and
    6 (dotted lines), and three versions of the PC algorithm:
    Pearson-PC (thin lines), $Q_n$-PC (medium lines) and RPC using
    Spearman's $\rho$ (thick lines).}
\end{figure}

\begin{figure}[t]
  \centering
  \includegraphics[scale=0.65]{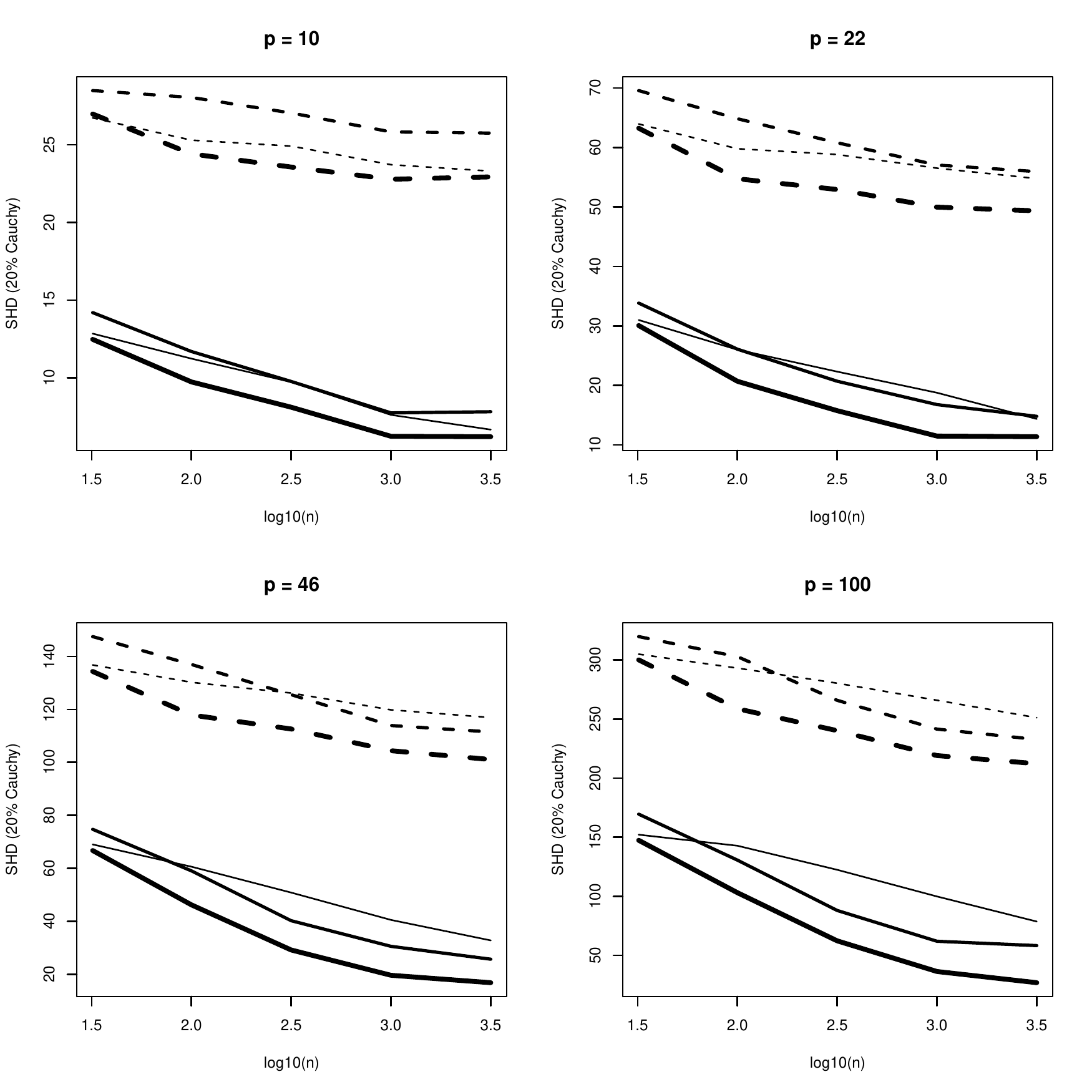} 
  \caption{  \label{fig:Cauchymixture}
    Structural Hamming distances for contaminated data, graphs with
    expected degree 3 (solid lines) and 6 (dotted lines), and three
    versions of the PC algorithm: Pearson-PC (thin lines), $Q_n$-PC
    (medium lines) and RPC using Spearman's $\rho$ (thick lines).}
\end{figure}

A clear message emerges from the plots.  First,
Figure~\ref{fig:normal} shows that for normal data, RPC performs only
marginally worse than Pearson-PC.  The $Q_n$-PC algorithm also does
well, although some gap in SHD arises for small sample sizes.  Second,
Figure~\ref{fig:F11} shows a dramatic gain in performance for RPC for
the Gaussian copula data with $F_{1,1}$ marginals.  In fact, the SHD
associated with the other two graph estimators is comparable to that
of estimating the graph to always be empty.  The expected SHD between
the empty graph and a graph on $p$ nodes with expected degree $d$ is
simply the expected number of edges in our random graphs, which is
$pd/2$.  For our choices of $d=3$ and $d=6$, the respective expected
SHD is 150 and 300 when $p=100$.  Finally,
Figure~\ref{fig:Cauchymixture} shows that RPC outperforms $Q_n$-PC for
the contaminated data; $Q_n$-PC outperforms Pearson-PC for larger
choices of $p$.


\section{Conclusion}
\label{sec:conclusion}

The PC algorithm of \cite{spirtes:2000} addresses the problem of model
selection in graphical modelling with directed graphs via a clever
scheme of testing conditional independences.  For multivariate normal
observations, the algorithm is known to have high-dimensional
consistency properties when conditional independence is tested using
sample partial correlations \cite{Kalisch:2007}.  We show that the PC
algorithm retains these consistency properties when observations
follow a Gaussian copula model and rank-based measures of correlation
are used to assess conditional independence.  The assumptions needed
in our analysis are no stronger than those in prior Gaussian work when
the considered sequence of DAGs has bounded degree.  When the degree
grows our assumptions are slightly more restrictive as our proof
requires control of the conditioning of principal submatrices of
correlation matrices that are inverted to estimate partial
correlations in the rank-based PC (RPC) algorithm.

Our simulations show that for normal data the RPC algorithm does
essentially as well as the sample correlation-based version of the
algorithm.  As can be expected, we see RPC retain this performance for
Gaussian copula data, for which sample correlations are poorly suited.
Somewhat surprisingly, RPC also performed better than a previously
considered robust version of the PC algorithm under a contamination
model.  We remark that the consistency theory available for this
robust version is for a fixed graph size $p$.  Since rank correlations
take only marginally longer to compute than sample correlations,
hardly any downsides are associated with making 
RPC the standard version of the PC algorithm for continuous data.  

In our work on consistency, the data-generating distribution is
assumed to be faithful to an underlying DAG.  In fact, our results
make the stronger assumption that non-zero partial correlations are
sufficiently far from zero.  As shown in \cite{Uhler:2012}, this can
be a restrictive assumption, which provides an explanation for why
consistency does not `kick-in' quicker in our simulation study.

Finally, we remark that extensions of the PC algorithm exist to deal
with situations in which some causally relevant variables remain
unobserved.  Such algorithms infer a more complex graphical object;
compare \cite{spirtes:2000} and \cite{colombo:2012}.  It is reasonable
to expect the use of rank correlations to be beneficial in those
settings as well, and a study of these algorithms would be an
interesting topic for future work.

\section*{Acknowledgments}

Mathias Drton was
supported by the NSF under Grant No.~DMS-0746265 and by an Alfred P. Sloan
Fellowship.

\begin{appendix}
\section{Sample size adjustment}

We now show that the consistency result in
Corollary~\ref{ConsistencyOfPC} still holds when using the conditional
independence tests from \eqref{KBCITest}.  In these tests, the sample
size is adjusted from $n$ to $n-|S|-3$.

\begin{proof}
  The test in \eqref{KBCITest} accepts a conditional independence
  hypothesis if and only if
  \begin{equation}\label{KBCITestEquiv}
    | \hat{\rho}_{uv \vert S} |\; \le\; \gamma(n,|S|,z), 
  \end{equation}
  where
  \begin{equation}\label{cutoffdef}
    \gamma(n,|S|, z) =
    \frac{\exp\big(z/\sqrt{n - |S| - 3}\big) -
    1}{\exp\big(z/\sqrt{n - |S| - 3}\big) + 1} 
  \end{equation}
  and $z = z(\alpha) = 2\Phi^{-1}(1 - \alpha/2)$.  We need to find a
  sequence $(\alpha_n)$ of values for $\alpha$ such that consistency
  holds under the scaling assumptions made in
  Corollary~\ref{ConsistencyOfPC}.  We will do this by specifying a
  sequence $(z_n)$ for values for the (doubled) quantiles $z$.
  
  
  We claim that the RPC algorithm using the tests
  from~(\ref{KBCITestEquiv}) is consistent when choosing the quantile
  sequence
  \begin{equation}
    \label{eq:z_n-choice}
  z_n = \sqrt{n - 3}\cdot
  \log\left(\frac{1 + c_n/3}{1 - c_n/3}\right),
  \end{equation}
  where we use the abbreviation
  \[
  c_n := c_{\min}(\Sigma_n,q_n).
  \]
  We will show that as the sample size $n$ tends to infinity, with
  probability tending to one, $| \hat{\rho}_{uv \vert S} -
    \rho_{uv \vert S} | < c_n/3$ for every $u, v \in V$ and $|S|
  \le q_n$.  Furthermore, we will show that for the above choice of
  $z_n$ and all sufficiently large $n$, we have $c_n/3 \le \gamma(n,
  |S|, z_n) \le 2 c_n/3$ for each relevant set $S$ with $0 \le | S|
  \le q_n$.  These two facts imply that, with asymptotic probability
  one, every conditional independence test is correct, and the RPC
  algorithm succeeds.
  
  First, we slightly adapt the proof of Theorem~\ref{ErrorProbsForPC}.
  Choosing the uniform error threshold for the correlation estimates
  as
  \begin{equation}\label{c3epsdef}
    \epsilon = \frac{c \lambda^2}{(6+c) q + \lambda c q} > 0
  \end{equation}
  in place of \eqref{epsdef} yields that, with probability at least
  \begin{equation}
    \label{eq:app:prob-bound}
  1 - \frac{A}{2} p^2 \exp \left(-\frac{B \lambda^4 n c^2}{64 q^2} \right),
  \end{equation}
  we have that $| \hat{\rho}_{uv \vert S} - \rho_{uv \vert S} | < c/3$
  for every $u, v \in V$ and $|S| \le q$.  When substituting $p_n$,
  $q_n$, $c_n$ and $\lambda_{\min}(\Sigma_n,q_n)$ for $p$, $q$, $c$
  and $\lambda$, respectively, the scaling assumptions in
  Corollary~\ref{ConsistencyOfPC} imply that the probability bound
  in~(\ref{eq:app:prob-bound}) tends to one as $n\to\infty$, and we
  obtain the first part of our claim.

  
  
  For the second part of our claim, note that our choice of $z_n$
  in~(\ref{eq:z_n-choice}) gives $\gamma(n,0,z_n) = c_n/3$.  Since
  $\gamma(n, |S|, z)$ is monotonically
  increasing in $|S|$, we need only show that for
  sufficiently large $n$,
  \begin{equation*}
    \gamma(n, q_n, z_n) - \gamma(n, 0, z_n) \le c_n/3.
  \end{equation*}
  For $x\ge 0$, the function
  \begin{equation*}
    f(x) = \frac{\exp(x) - 1}{\exp(x) + 1}
  \end{equation*}
  is concave and, thus, for any $q_n\ge 0$,
  \begin{align}
    \nonumber
    \gamma(n, q_n, z_n) - \gamma(n, 0, z_n) &=
    f\left(\frac{z}{\sqrt{n - q_n - 3}}\right) -
    f\left(\frac{z}{\sqrt{n - 3}}\right) \\  
    \label{eq:linear-bound}
    &\le f'\left(\frac{z}{\sqrt{n - 3}}\right) \left(\frac{z}{\sqrt{n
          - q_n - 3}} - \frac{z}{\sqrt{n - 3}}\right).
  \end{align}
  The derivative of $f$ is
  \[
  f'(x) = \frac{2\exp(x)}{\left(\exp(x) + 1\right)^2}.
  \]
  Evaluating the right hand side of~(\ref{eq:linear-bound}), we obtain
  that
  \begin{align}
    \nonumber
    \gamma(n, q_n, z_n) - \gamma(n, 0, z_n) 
    &\le   \frac{1}{2}\left(1-\frac{c_n^2}{9}\right)
    \log\left(\frac{1 + c_n/3}{1 - c_n/3}\right)
    \left(\frac{\sqrt{n-3}}{\sqrt{n - q_n - 3}} - 1\right) \\ 
    &\le \frac{1}{2} \log\left(\frac{1 + c_n/3}{1 - c_n/3}\right)
    \left(\frac{\sqrt{n-3}}{\sqrt{n - q_n - 3}} - 1\right).
    \label{Gap}
  \end{align}
  Being derived from absolute values of partial correlations, the
  sequence $c_n$ is in $[0,1]$.  Now, $\log[(1 + x)/(1-x)]$ is a
  convex function of $x\ge0$ that is zero at $x=0$ and equal to
  $\log(2)$ for $x=1/3$.  Therefore,
  \[
  \frac{1}{2}\log\left(\frac{1 + c_n/3}{1 - c_n/3}\right) \;\le\; 
  \frac{1}{2}\log(2) \cdot c_n, \qquad c_n\in[0,1].
  \]
  This shows that the bound in~(\ref{Gap}) is $o(c_n)$ because, by
  assumption, $q_n=o(\sqrt{n})$.  In particular, the bound
  in~\eqref{Gap} is less than $c_n/3$ for sufficiently large $n$,
  proving the claimed consistency result.
\end{proof}
\end{appendix}

\bibliographystyle{amsalpha} 
\bibliography{copula_pc}

\end{document}